\documentclass[a4paper,10pt]{article}
\usepackage[utf8]{inputenc}

\usepackage{amsmath,amsthm,amssymb}

\newtheorem{theorem}{Theorem}
\newtheorem{lemma}[theorem]{Lemma}
\newtheorem{corollary}[theorem]{Corollary}

\begin{document}
\title{Ordinal Notations in Caucal Hierarchy}
\author{Fedor~Pakhomov\thanks{This work was partially supported by RFFI grant 15-01-09218 and Dynasty foundation.}\\Steklov Mathematical Institute,\\Moscow\\ \texttt{pakhfn@mi.ras.ru}}
\date{December 2015}
\maketitle

\begin{abstract} Caucal hierarchy is a well-known class of graphs with decidable monadic theories. It were proved by L.~Braud and A.~Carayol that well-orderings in the hierarchy are the well-orderings with order types less than $\varepsilon_0$. Naturally, every well-ordering from the hierarchy could be considered as a constructive  system of ordinal notations.  In proof theory  constructive  systems of ordinal notations with fixed systems of cofinal sequences are used for the purposes of classification of provable recursive functions of theories. We show that any well-ordering from the hierarchy could be extended by a monadically definable system of cofinal sequences with Bachmann property. We show that the growth speed of functions from fast-growing hierarchy based on constructive ordinal notations  from Caucal hierarchy may be only slightly influenced by the choice of monadically definable systems of cofinal sequences.  We show that for ordinals less than $\omega^\omega$ a fast-growing hierarchy based on any system of ordinal notations from Caucal hierarchy coincides with Löb-Wainer hierarchy.
\end{abstract}

\section{Introduction}
We are interested in the hierarchy of (possibly infinite) graphs  that were introduced by D.~Caucal \cite{Cau96}. Graphs in the hierarchy are directed graphs with colored edges. They  naturally  could be considered as structures with several binary relations (one relation for each color). All the graphs in the hierarchy have decidable monadic theories. 

There are several characterizations of this hierarchy (we refer to the survey \cite{Ong07}). We will use definition  due to Caucal \cite{Cau02,CarWoh03}. The level zero of hierarchy consists of all finite graph. Each next level of hierarchy  consists of all graphs that are monadically interpretable in unfoldings of graphs of the previous level. 

 Caucal hierarchy contains wide spectrum of  structures. In our investigation we focus on well-orderings that are monadically definable in some graphs from the hierarchy. L.~Braud has shown that for every $\alpha<\varepsilon_0=\lim_{n\to\omega}\underbrace{\omega^{\dots^{\omega}}}_{\mbox{\scriptsize $n$ times}}$ there exists a well-ordering of order type $\alpha$ in Caucal hierarchy \cite{Bra09}. L.~Braud and A.~Carayol have shown that the converse is true and any well-ordering that lies in the hierarchy have order type less than $\varepsilon_0$.

The ordinal $\varepsilon_0$ is known to be proof-theoretic ordinal of first-order arithmetic (for a survey of ordinals in proof theory we refer to \cite{Rat07}). Natural fragments of first-order arithmetic have proof-theoretic ordinals less than $\varepsilon_0$. Constructive ordinal notation systems are used within the context of study of proof-theoretic ordinals. The standard general method of defining constructive ordinal notations is Kleene $\mathcal{O}$ \cite{Kln38}\cite{Chr38}. It is known that wast number of proof-theoretic applications is sensitive to the choice of constructive representations of ordinals. There is a long-standing conceptual problem of determining what is a {\it natural} or {\it canonical} ordinal notation system \cite{Kre76,Fef96}. 

One of the uses of ordinals in proof theory are classifications of provably-recursive functions of theories. There are several  approaches to define fast-growing recursive functions using ordinal notations \cite{Ros84}.  There are number of connections between this different methods of constructing fast-growing functions. Most of the approaches essentially uses systems of cofinal sequences for ordinals. A cofinal sequence (of the length $\omega$) for an ordinal $\alpha$ is a sequence $\alpha_0<\alpha_1<\ldots$  such that every $\alpha_i<\alpha$ and $\lim_{n\in \omega}\alpha_n=\alpha$. Let us is considered some proper initial segment of countable ordinals. A system of cofinal sequences for ordinals in the segment is some assignment of cofinal sequences (of length $\omega$) for all limit ordinals in the segment. In this paper we focus on one of the methods to construct fast-growing functions from ordinals with systems of cofinal sequences: fast-growing hierarchy. 

Because monadic theories of graphs in Caucal hierarchy are decidable and in every graph from the hierarchy all the vertices are monadically definable, if we have a monadically definable well-ordering in some graph from the hierarchy then we can straightforward build a constructive system of ordinal notations from it. In the present paper we investigate how can we naturally define system of cofinal sequences for well-orderings from the hierarchy. We will represent systems of cofinal sequences by binary predicates $x R y$ that means ``$x$ is in an element of the cofinal sequence for $y$''; in this way we can talk about definability of systems of cofinal systems in structures. We show that if there is a graph with monadically definable well-ordering in $n$-th level of Caucal hierarchy then in the same level there is an extension of this graph by new colors and edges marked by the new colors such that some system of cofinal sequences is monadically definable in the extension. Moreover, for a deterministic tree from Caucal hierarchy there is always a monadically definable system of cofinal sequences for a monadically definable well-ordering; note that every graph in Caucal hierarchy could be monadically interpreted in a deterministic tree from  the same level of Caucal hierarchy \cite{CarWoh03}.

There is Bachmann property for systems of cofinal sequences that guarantees that hierarchies of functions based on them behave relatively well \cite{Bach67,Sch77}. We show that if there is a graph with monadically definable well-ordering and system of cofinal sequences for it in $n$-th level of Caucal hierarchy then in the same level there is an extension of this graph by new colors and edges of the new colors such that there is a monadically definable system of cofinal sequences with Bachmann property that contains only subsequences of the original cofinal sequences. Moreover, for a deterministic tree from Caucal hierarchy the same could be done without switching to an extension.

We consider graphs from Caucal hierarchy with monadically-definable well-ordering $\prec$ and two monadically definable systems of cofinal sequences with Bachmann property. Two fast-growing hierarchies of recursive functions arises: $F^1_{a}(x)$ and $F^2_{a}(x)$. We show that if $a \prec b$ then $F^1_b(x)$  eventually dominates  $F^2_a(x)$ and $F^2_b(x)$  eventually dominates  $F^1_a(x)$, i.e. hierarchies are very close to each other.

Let us consider a graph from Caucal hierarchy with monadically-definable well-ordering $\prec$ and monadically definable system of cofinal sequences with Bachmann property. We have  fast-growing hierarchy of recursive functions based on this well-ordering and system of cofinal sequences: $F'_a(x)$. Let us denote by $f$ the embedding of the well-ordering onto initial segment of ordinals. Most wide-known concrete fast-growing hierarchy up to $\varepsilon_0$  is Löb-Wainer hierarchy \cite{LobWai70} we denote the functions from it as $F_{\alpha}(x)$.  We show that for all $\alpha<\beta<\omega^\omega$, if $f^{-1}(\beta)$ is defined then $F'_{f^{-1}(\beta)}(x)$  eventually dominates  $F_\alpha(x)$ and $F_\beta(x)$ eventually dominates  $F'_{f^{-1}(\alpha)}(x)$.

\section{Caucal Hierarchy of Graphs}
Caucal hierarchy is an hierarchy of directed graphs with colored edges. Levels of the hierarchy are indexed by natural numbers. Formally, {\it directed graph with colored edges} is a tuple $(\mathbf{C},V,U)$, where $\mathbf{C}$ is a set of edge colors, $\mathbf{C}$ is finite, $V$ is a set of vertices, and  $U$ is a set of edges,  $U\subset V\times\mathbf{C}\times V$. Because directed graphs with colored edges are the only graphs that we consider, if we use the term graph it will refer to directed graphs with colored edges.    We say that a triple $(v_1,c,v_2)\in V\times\mathbf{C}\times V$ is an {\it edge of color} $c$ from $v_1$ to $v_2$; we also say that the edge $(v_1,c,v_2)$ is marked by color $c$ .

We denote the set of all graphs from the $n$-th level of Caucal hierarchy by $\mathbf{Graph}_n$. We will have $\mathbf{Graph}_n\subset\mathbf{Graph}_{n+1}$ for all $n$.  $\mathbf{Graph}_0$ is the collection of all finite graphs (formally, in axiomatic set theory we can't consider such a set; we could overcome this difficulties, for example as following: instead of considering all finite graphs we can consider only graphs, where $\mathbf{C}$ and $V$ are hereditary countable sets). 

Suppose we have a graph $G=(\mathbf{C},V,U)$ and a vertex $v_0$. The {\it unfolding} of $G$ from $v_0$ is the graph $(\mathbf{C},P,R)$, where $P$ is the set of all sequences $(p_0,a_0,p_1,a_1\ldots,a_{n-1},p_{n})$ such that  $n\ge 0$, $p_i\in V$, $a_i\in U$, and every $a_i$ is of the form $(p_i,x,p_{i+1})$ and $R$ consist of all edges $$((p_0,a_0,p_1,a_1\ldots,a_{n-1},p_{n}),c,(p_0,a_0,p_1,a_1\ldots,a_{n-1},p_{n},(p_n,c,p_{n+1}),p_{n+1}))$$ such that $$(p_0,a_0,p_1,a_1\ldots,a_{n-1},p_{n}),(p_0,a_0,p_1,a_1\ldots,a_{n-1},p_{n},(p_n,c,p_{n+1}),p_{n+1})\in P.$$ We denote by $\mathcal{U}(G,v_0)$ the result of unfolding of $G$ from $v_0$.

Graph $(\mathbf{C},V,U)$ could be treated as the structure with the domain $V$, signature of binary predicates $\{\textsc{R}_c\mid c\in \mathbf{C}\}$, and the following interpretations of the predicates: 
$$v_1 \textsc{R}_c v_2 \stackrel{\text{def}}{\iff} (v_1,c,v_2)\in U.$$
Vice versa, for a finite set $\mathbf{C}$ a structure $\mathfrak{A}$ with the domain $A$ and the signature consists of binary predicates $\{\textsc{R}_c\mid c\in \mathbf{C}\}$ could be treated as a graph with the set of colors $\mathbf{C}$, the set of vertices $A$, and the set of edges $\{(v_1,c,v_2) \mid v_1,v_2\in A\textrm{ and }\mathfrak{A}\models v_1 \textsc{R}_c v_2\}$. 
We freely switch between this two formalisms. 

{\it Monadic second-order formulas} are formulas in which in addition to normal first-order connectivities and quantifiers, unary predicate variables and quantifiers over unary predicates are allowed. We use capital Latin letters  $X,Y,Z,\ldots$ for predicate variables and small Latin letters $x,y,z,\ldots$ for first-order variables.   The truth of monadic formulas in a structure $\models_{\mathrm{MSO}}$ could be defined in a natural way with unary predicate variables ranging over arbitrary subsets of the domain of the structure. Suppose $\mathfrak{A}$ and $\mathfrak{B}$ are structures with predicate-only signatures.

In a natural way for a structure $\mathfrak{A}$ with the domain $A$ we can talk about {\it monadically definable} subsets of $$(\mathcal{P}(A))^n\times A^m,$$ for each $n,m$. A set $$B\subset (\mathcal{P}(A))^n\times A^m$$ is monadically definable if there is a monadic formula $\varphi(X_1,\ldots, X_n,x_1,\ldots x_m)$ such that for every $(P_1,\ldots,P_n,p_1,\ldots,p_m)\in (\mathcal{P}(A))^n\times A^m $ we have
$$(P_1,\ldots,P_n,p_1,\ldots,p_m)\in B \iff \mathfrak{A}\models_{\mathrm{MSO}}\varphi(P_1,\ldots,P_n,p_1,\ldots,p_m).$$

 A {\it monadic interpretation} of a structure $\mathfrak{A}$ in a structure $\mathfrak{B}$ is a function $f$ that maps symbols from the signature of $\mathfrak{A}$ to monadic formulas of the signature of $\mathfrak{B}$ such that for an $n$-ary symbol $R$ from the signature of $\mathfrak{A}$ the formula $f(R)$ is of the form $F(x_1,\ldots,x_n)$ and there exists a bijection $g$ from the domain of $\mathfrak{A}$ to the domain of $\mathfrak{B}$ with the following holds for any $n$, $n$-ary symbol $R$ from the signature of $\mathfrak{A}$, and $v_1,\ldots,v_n$ from the domain of $\mathfrak{A}$:
$$\mathfrak{A}\models_{\mathrm{MSO}} R(v_1,\ldots,v_n)\iff \mathfrak{B}\models_{\mathrm{MSO}} f(R)(g(v_1),\ldots,g(v_n)).$$    
For graphs $G_1,G_2$, a {\it monadic interpretation} of $G_1$ in $G_2$ is a monadic interpretation of the graph $G_1$ as the structure in the graph $G_2$ as the structure.

Suppose we have defined the class $\mathbf{Graph}_n$. Then the class $\mathbf{Tree}_{n+1}$ consists of all unfoldings of graphs from $\mathbf{Graph}_n$ from arbitrary fixed vertices.  The class $\mathbf{Graph}_{n+1}$ consists of all graphs that can be monadically interpreted in the graphs from $\mathbf{Tree}_{n+1}$.

For a finite set of colors $\mathbf{C}$ we denote by $\mathbf{C}^{-}$ the set of fresh colors $c^{-}$ for all $c\in\mathbf{C}$. We can extend every graph $G=(\mathbf{C},V,U)$ to the graph $\mathcal{R}(G)=(\mathbf{C}\sqcup\mathbf{C}^{-},V,U\sqcup U^{-})$, where $U^{-}=\{(v_2,c^{-},v_1)\mid (v_1,c,v_2)\in U\}$.  It is easy to see that if $G\in \mathbf{Graph}_n$ then $\mathcal{R}(G)\in\mathbf{Graph}_n$.

For a graph $(\mathbf{C},V,E)$, vertices $v_1,v_2\in V$, and word $c_1c_2\ldots c_k\in \mathbf{C}^{\star}$ we say that there is $c_1c_2\ldots c_k$-marked path if there ares edges:
 $$(w_0,c_1,w_1),(w_1,c_2,w_2),\ldots,(w_{k-1},c_k,w_k)\in E.$$

We call a graph $G$ {\it deterministic} if for each color and vertex there exists at most one edge of this color going from the vertex.   
\begin{theorem}\label{Inv_rat_theorem}(\cite{CarWoh03}) For every $G=(\mathbf{C},V,U)\in \mathbf{Graph}_n$ there exists deterministic tree $T=(\mathbf{D},V,E)\in \mathbf{Tree}_n$  and regular languages $L_c\subset(\mathbf{D}\sqcup\mathbf{D}^{-})^{\star}$, for $c\in \mathbf{C}$, such that for all $v_1,v_2\in V$ we have $(v_1,c, v_2)\in U$ iff there exists a path from $v_1$ to $v_2$ in $\mathcal{R}(T)$ marked by some $\alpha\in L_c$.
\end{theorem}

Note that if some $G$ and $T$ are in a relation as in Theorem \ref{Inv_rat_theorem} then there is  a monadic interpretation of $G$ in $T$ \cite{CarWoh03}.

\section{Higher Order Pushdown Automatons}
One of alternative ways to represent Caucal hierarchy are configuration graphs of higher order pushdown automatons. There is a pumping lemma for graphs in Caucal hierarchy formulated in the terms of  this representation \cite{Par12}. We need this pumping lemma for the proofs in Section \ref{fast_growing_section}.   Higher order pushdown automatons were introduced by Maslov \cite{Mas74}

Suppose $\mathbf{A}$ is a finite set. We are going to define the notion of {\it higher order pushdown store} ({\it pds})  over an alphabet $\mathbf{A}$. A $0$-pds over $\mathbf{A}$ is just a symbol from $\mathbf{A}$. An $(n+1)$-pds is a finite sequence of $n$-pds's. 

Suppose $n>m$, $\alpha^n$ is an $n$-pds and $\alpha^m$ is an $m$-pds. We will define $n$-pds $\alpha^n{:}\alpha^m$. If $n=m+1$ then $\alpha^n{:}\alpha^m$ is the result of attaching $\alpha^m$ at the end of the sequence $\alpha^n$. If $n>m+1$ then  $\alpha^n{:}\alpha^m$ is equal to $\alpha^n{:}\beta^{m+1}$, where $\beta^{m+1}$ is $m+1$-pds that is an one element sequence containing $\alpha^m$.

We say that every $0$-pds is {\it proper}. We call an $n+1$-pds proper if it is a non-empty sequence of proper $n$-pds's. 

Note that for every $n$ and $m<n$, a proper $n$-pds $\alpha^n$ there is the unique representation of $\alpha^n$ in the form $\beta^n{:}\beta^m$, where $\beta^m$ is $m$-pds; it is easy to see that  here $\beta^m$ is always proper.

 Note can represent every proper $n$-pds $\alpha$ in the form $\beta^n{:}(\beta^{n-1}{:}(\ldots (\beta^1{:}\beta^0)\ldots))$ in the unique way. We say that $\beta^k{:}(\beta^{k-1}{:}(\ldots (\beta^1{:}\beta^0)\ldots))$ is the {\it topmost} $k$-pds of $\alpha$ .

There are the following operations on proper $n$-pds's over an alphabet $\mathbf{A}$:
\begin{enumerate}
\item Operation $\textsf{pop}^k$, where $0< k\le n$. This operation  maps a proper $n$-pds $\alpha^n{:}\alpha^{k-1}$ to $n$-pds $\alpha^n$.
\item Operation $\textsf{push}^k(a)$, where $0< k\le n$ and $a\in\mathbf{A}$. This operation transforms a proper $n$-pds by replacing its topmost $k$-pds $\alpha^k{:}\alpha^{k-1}$ with $(\alpha^k{:}\alpha^{k-1}){:}\beta^{k-1}$, where $\beta^{k-1}$ is $\alpha^{ k-1}$ with topmost $0$-pds replaced by $a$.
\end{enumerate}
We denote by $\mathbf{OP}^n(\mathbf{A})$ the set of all operations on $n$-pds's over alphabet $\mathbf{A}$ defined above.

A pushdown system of level $n$ is a tuple $\mathcal{A}=(\mathbf{A},\mathbf{S},s_I,Q,q_I,\Delta,\lambda)$:
\begin{enumerate}
\item $\mathbf{A}$ is a finite input alphabet,
\item $\mathbf{S}$ is a finite alphabet of stack symbols,
\item $s_I$ is an initial stack symbol,
\item $Q$ is a finite set of automaton states,
\item $q_I$ is an initial state,
\item $\Delta\subset Q\times \mathbf{S}\times Q\times\mathbf{OP}^n(\mathbf{S})$ is a set of transition instructions,
\item $\lambda\colon \Delta\to \mathbf{A}\sqcup \{\varepsilon\}$ is a labeling function.
\end{enumerate}

A {\it configuration} of pushdown system $\mathcal{A}$ is a pair $(q,\alpha)$, where $q\in Q$ and $\alpha$ is a proper $n$-pds. The {\it initial configuration} is the pair $(q_I,\varepsilon^n{:}(\varepsilon^{n-1}{:}(\ldots {:}(\varepsilon^1{:}s_I)\ldots )))$, where all $\varepsilon^{k}$ are empty $k$-pds's.

For configurations $w_1=(q_1,\alpha_1)$ and $w_2=(q_2,\alpha_2)$ we say that $\mathcal{A}$ {\it admits an one step transition} from $w_1$ to $w_2$ if the topmost $0$-pds of $\alpha_1$ is $s_1$, there exists $o\in\mathbf{OP}^n(\mathbf{S})$ such that the result of application of $o$ to $\alpha_1$ is $\alpha_2$, and $(q_1,s_1,q_2,o)\in\Delta$. Note that for a pair of configurations there exists at most one transition instruction $(q_1,s_1,q_2,o)$ with this property; we say that $(q_1,s_1,q_2,o)$ is the transition instruction that makes a switch from configuration $w_1$ to configuration $w_2$.

We say that there is an $\mathcal{A}$  {\it run} from configuration $w$ to configuration $w'$ if we ca find  configurations $u_0,u_1,\ldots, u_k$ such that $\mathcal{A}$ admits an one step transition from $u_i$ to $u_{i+1}$ for all $i<k$, $u_0=w$, and $u_k=w'$. 

A configuration $w$ is called {\it reachable} for $\mathcal{A}$ if there exists a run from the initial configuration of $\mathcal{A}$ to $w$. 

{\it Configuration graph} $C(\mathcal{A})=(\mathbf{A}\sqcup\{\varepsilon\},V_{C(\mathcal{A})},E_{C(\mathcal{A})})$ of $\mathcal{A}$ is the graph, where $V_{C(\mathcal{A})}$ is the set of all reachable configurations of $\mathcal{A}$ and $E_{C(\mathcal{A})}$ is the set of all triples $(w_1,\lambda(t),w_2)$, where $w_1,w_2$ are reachable configurations of $\mathcal{A}$, $\mathcal{A}$ admits an one step transition from $w_1$ to $w_2$, and $t$ is the transition instruction that makes a switch from configuration $w_1$ to configuration $w_2$.

Let us consider next only the case of  graphs $C(\mathcal{A})$ such that for every vertex either all outgoing edges are marked by $\varepsilon$ or all of them aren't marked by $\varepsilon$ (we will talk about $\varepsilon$-contractions only in the case of $\mathcal{A}$ with this property).
We define $G=(\mathbf{C},V,U)$  the {\it $\varepsilon$-contraction} of a configuration graph $C(\mathcal{A})$. $V$ consists of the initial configuration of $\mathcal{A}$ and all configurations that have incoming edges marked by non-$\varepsilon$ color. There is a $c$-marked edge in $G$ between $v_1,v_2\in V$ iff in $C(\mathcal{A})$ there is an $\varepsilon$-marked path from $v_1$ to some $v_3$ such that there is  $c$-marked edge from $v_3$ to $v_2$

We denote by $\mathbf{HOPDG}_n$  the collection of all graphs that are isomorphic to the $\varepsilon$-contraction of the configuration graph of some $n$-pds.

\begin{theorem}(\cite{CarWoh03})\label{HOPDG} For every  $n$ and every graph $G$ such that there is a vertex and paths from it  to any other vertex  we have $G\in\mathbf{Graph}_n$ iff $G\in\mathbf{HOPDG}_n$
\end{theorem}

P.~Parys\cite{Par12} have proved a variant of pumping lemma for graphs in $\mathbf{HOPDG}$.

Functions $\beth_n\colon \omega\to \omega$ are defined as the following:
\begin{itemize}
\item $\beth_0(x)=x$,
\item $\beth_{n+1}(x)=2^{\beth_n(x)}$.
\end{itemize}

\begin{theorem}(\cite{Par12}) \label{pumping_lemma}
Suppose $\mathcal{A}$ is $n$-pds with input alphabet $\mathbf{A}$ and $L\subset \mathbf{A}^{\star}$ is a regular language. Let $G$ be the $\varepsilon$-contraction of  the configuration graph of $\mathcal{A}$. Assume that $G$ is finitely branching.  Assume that in $G$ there exists a path of length $m$ from the initial configuration to some configuration $w$. Assume that for some $\alpha\in L$, $|\alpha|\ge\beth_{n-1}((m+1)C_{\mathcal{A},L})$ there is path in $G$ from $w$ marked by $\alpha$; here $C_{\mathcal{A},L}$ is a constant which depends on $\mathcal{A}$ and $L$. Then there are infinitely many paths in $G$, which start in $w$  and are marked by some $\beta\in L$.
\end{theorem}





\section{Well-Orderings and Cofinal Sequences}
We are interested in graphs $G=(\mathbf{C},V,U)$ such that one of $R_c$, for $c\in\mathbf{C}$ is a strict well-ordering of $V$. We call such graphs {\it well-orderings for color $c$}.

 L.~Braud have proved that there are graphs in Caucal hierarchy with well-ordering of order type $\alpha$, for every $\alpha<\varepsilon_0$ \cite{Bra09}. L.~Braud and A.~Carayol have proved that every graph with well-ordering in Caucal hierarchy has order type less that $\varepsilon_0$ \cite{BraCar10}. More precisely, those works have shown that for a class $\mathbf{Graph}_n$ the exact non-reachable upper bound of well-orderings order types is $\omega_{n+1}$, where $\omega_0=1$ and $\omega_{k+1}=\omega^{\omega_k}$.

 Suppose $(A,R)$ is a strict well-ordering. We denote by $A^{\mathrm{lim}}$ the set of all non-zero limit points of $(A,R)$, i.e. $A^{\mathrm{lim}}=\{x\in A\mid x=\sup_R(\{y\in A\mid yR x\})\}\setminus \{\inf_R(A)\}$. A {\it system of cofinal sequences} is a function $s\colon A^{\mathrm{lim}}\times \omega  \to A$ such that for all $x\in A^{\mathrm{lim}}$ and $n\in \omega$ we have $s(x,n)R x$ and $x=\sup_R(\{s(x,m)\mid m\in \omega\})$. For the purposes of defining cofinal sequences in the terms of monadic theories of graphs we further consider only strictly monotone cofinal sequences, i.e. for every $x\in A^{\lim}$ if $n<m$ then $s(x,n)R s(x,m)$.  

A system of cofinal sequences $s\colon A^{\mathrm{lim}}\times \omega  \to A$  is said to have {\it Bachmann property} \cite{Bach67} if for any $x,y\in A^{\mathrm{lim}}$ and $n\in \omega$ such that $s(x,n)<y\le s(x,n+1)$ we have $s(x,n)\le s(y,0)$.

There is an alternative characterization of systems of cofinal sequences with Bachmann property (we refer to \cite[\S3.3]{Ros84} for details). System of cofinal sequences gives step-down relation $\prec^s$ on $A$. The relation $\prec^s$ is the transitive closure of $x \prec^s x+1$ and $s(x,0)\prec^s x$. A system of strictly monotone cofinal sequences $s\colon A^{\mathrm{lim}}\times \omega  \to A$ is said to be Schmidt-coherent \cite{Sch77} if for all $x\in A^{\mathrm{lim}}$ and $n\in \omega$ we have $s(x,n)\prec^ss(x,n+1)$.    

Suppose $G=(\mathbf{C},V,U)$ is a well-ordering for color $c$ and $d\in\mathbf{C}$. We call say that $d$ {\it gives a system of cofinal sequences for $R_c$} if there is a system of strictly monotone cofinal sequences $s_d$ for $(V,R_c)$ such that for all $v_1,v_2\in V$
$$v_1 R_d v_2\iff v_2\textrm{ is non-zero limit of $R_c$ and } \exists n ( s_d(v_2,n)=v_1) .$$
Note that there exists at most one $s_d$ with described property and thus it is determined by $R_d$.

For a well-ordering $(A,R)$ with a fixed system of cofinal sequences $s$ the {\it fast-growing hierarchy} is a family of functions $F_a\colon \omega \to \omega$ indexed by $a\in A$:
\begin{enumerate}
\item $F^s_{0}(x)=x+1$,
\item $F^s_{b+1}(x)=\underbrace{F_{b}(F_{b}(\ldots F_{b}}\limits_{\mbox{$x$-times}}(x)\ldots))$,
\item $F^s_l(x)=F_{s(l,x)}(x)$, for every $l\in A^{\mathrm{lim}}$.
\end{enumerate}

\section{Defining Cofinal Sequences in Caucal hierarchy}
In the present section we will prove that every $G\in \mathbf{Graph}_n$ that is a well-ordering for some color $c$ could be extended to $G'\in \mathbf{Graph}_n$ by adding color $d$ and some edges marked by $d$ such that $d$ gives a system of cofinal sequences for $R_c$. Note that equivalently Theorem \ref{adding_cofinals} and Theorem \ref{making_Bachmann} could be formulated in the terms of monadic definable relations on deterministic trees from $\mathbf{Tree}_n$.

We consider finite automatons as tuples $(\mathbf{A},Q,\Delta)$, where $\mathbf{A}$ is input alphabet, $Q$ is the set of states, and $\Delta\subset Q\times \mathbf{A}\times Q$ is the set of instructions. 

 Suppose we have a tree $T=(\mathbf{D},V,E)$, and a finite automaton $\mathcal{A}$ with input language $\mathbf{D}\sqcup\mathbf{D}^{-}$ and the set of states $Q_{\mathcal{A}}$. We call the the set $\mathbf{Tp}_{\mathcal{A}}=\mathcal{P}(Q_{\mathcal{A}}\times Q_{\mathcal{A}})\times \mathcal{P}(Q_{\mathcal{A}}\times Q_{\mathcal{A}})$ the set of {\it vertex pair types}.  For $q_1,q_2\in Q_{\mathcal{A}}$ and $v_1,v_2\in V$ we say that $\mathcal{A}$ can {\it switch} from $q_1$ to $q_2$ on a path from $v_1$ to $v_2$ if there exists a path in $\mathcal{R}(T)$ from $v_1$ to $v_2$ marked by $\alpha\in (\mathbf{D}\sqcup\mathbf{D}^{-})^{\star}$  such that $\mathcal{A}$ can switch from $q_1$ to $q_2$ on input $\alpha$. We say that $(v_1,v_2)\in V\times V$ has a type $(A,B)$ if for a  $(q_1,q_2)\in Q_{\mathcal{B}}\times Q_{\mathcal{B}}$ we have 
\begin{enumerate}
\item $(q_1,q_2)\in A$ iff  $\mathcal{A}$ can switch from $q_1$ to $q_2$ on a path from $v_1$ to $v_2$,
\item $(q_1,q_2)\in B$  iff  $\mathcal{A}$ can switch from $q_1$ to $q_2$ on a path from $v_2$ to $v_1$.
\end{enumerate}

\begin{lemma} Suppose $\mathbf{D}$ is a finite set of colors and $\mathcal{A}$ is a finite automaton with input alphabet $\mathbf{D}\sqcup\mathbf{D}^{-}$. Then for every vertex pair types $t\in\mathbf{Tp}_{\mathcal{A}}$  there exists a monadic formulas $\varphi_t(x,y)$  such that for every tree $T=(\mathbf{D},V,E)$ and $v_1,v_2\in V$ we have $T\models_{\mathrm{MSO}} \varphi_t(v_1,v_2)$ iff $(v_1,v_2)$ have the type $t$.  
\end{lemma} 
\begin{proof}Suppose $k=|Q_{\mathcal{A}}|$ and we have enumerated all the elements of $Q_{\mathcal{A}}$. Clearly it is enough to prove that for any state $q\in Q_{\mathcal{A}}$ there exists a monadic formula $\psi_{q}(x,Y_1,\ldots, Y_k)$ such that for every tree $T=(\mathbf{D},V,E)$, $v\in V$, and $(S_1,\ldots,S_k)\in (\mathcal{P}(V))^k$ we have: $T\models_{\mathrm{MSO}} \psi_q(v,S_1,\ldots,S_k)$ iff, for all $1\le i \le k$, $S_i$ is the set of all vertices $v'$ such that $\mathcal{A}$ can switch from the state $q$ to $i$-th state on path from $v$ to $v'$.

 Suppose  $T=(\mathbf{D},V,E)$ is a tree, $v_1\in V$, and $q\in Q_{\mathcal{A}}$. Let us consider the tuple $(R_1,\ldots,R_k)\in (\mathcal{P}(V))^k$, where $R_i$ is the set of all vertices $v'$ such that $\mathcal{A}$ can switch from the state $q$ to the $i$-th state. There is a natural partial order on $(\mathcal{P}(V))^k$ by element-wise inclusion.  Note that $(R_1,\ldots,R_k)$ is the least element $(S_1,\ldots,S_k)\in (\mathcal{P}(V))^k$ such that, $v\in S_i$ if $q$ is $i$-th state and for all $1\le i,j\le k$,  $v'\in S_i$, and edge $(v',a,v'')$ from $\mathcal{R}(T)$, if $\mathcal{A}$ can switch from $i$-th state to $j$-th state on input $a$ then $v'\in S_j$. The late characterization of  $(R_1,\ldots,R_k)$ clearly gives us the required monadic formula $\psi_{q}$.
\end{proof}

Obviously the following two lemmas holds:
\begin{lemma} \label{type_calculation1} Suppose $\mathbf{D}$ is a finite set of colors and $\mathcal{A}$ is a finite automaton with input alphabet $\mathbf{D}\sqcup\mathbf{D}^{-}$. Then there exists a function $f\colon \mathbf{Tp}_{\mathcal{A}}\times  \mathbf{Tp}_{\mathcal{A}} \to \mathbf{Tp}_{\mathcal{A}}$  with the following property: Suppose  $T=(\mathbf{D},V,E)$ is a tree, $v_1,v_2,v_3$ are  vertices  such that $v_1$ is in the cone of $v_2$ and $v_3$ is not in the cone of $v_2$. Then the type $t_3$ of $(v_1,v_3)$ is equal to $f(t_1,t_2)$, where  $t_1$ is the type of $(v_1,v_2)$, $t_2$ is the type of $(v_2,u_3)$.
\end{lemma}
\begin{lemma} \label{type_calculation2} Suppose $\mathbf{D}$ is finite set of colors and $\mathcal{A}$ is a finite automaton with input alphabet $\mathbf{D}\sqcup\mathbf{D}^{-}$. Then there exists a function $f\colon \mathbf{Tp}_{\mathcal{A}}\times \mathbf{Tp}_{\mathcal{A}} \times \mathbf{Tp}_{\mathcal{A}} \to \mathbf{Tp}_{\mathcal{A}}$  with the following property: Suppose  $T=(\mathbf{D},V,E)$ is a tree, $v_1,v_2$ are  vertices  with non-intersecting cones,  $u_1$ is a vertex from the cone of $v_1$ and  $u_2$ is a vertex from the cone of $v_2$. Then the type $t_4$ of $(u_1,u_2)$ is equal to $f(t_1,t_2,t_3)$, where  $t_1$ is the type of $(v_1,v_2)$, $t_2$ is the type of $(v_1,u_1)$, and $t_3$ is the type of $(v_2,u_2)$.
\end{lemma}

Suppose  $T$ is a tree.  For a vertex $v$ of $T$ we call the set of all vertices reachable from $v$ (including $v$ itself) {\it $T$-cone} under $v$ . We say that a vertex $v$ is a {\it $T$-successor} of $v'$ if there is an edge from $v'$ to $v$ in $T$.

Suppose $(A,R)$ is a well-ordering and $a\in A^{\mathrm{lim}}$. We say that a set $B\subset A$ is {\it cofinal} in $a$ if $\sup_R B=a$.

\begin{theorem} \label{adding_cofinals} Suppose $n$ is a number, $G=(\mathbf{C},V,U)\in \mathbf{Graph}_n$ is a well-ordering for color $c\in \mathbf{C}$. Then there exists a color $d\not \in \mathbf{C}$, a set $U_d\subset V\times \{d\}\times V$ such that $G'=(\mathbf{C}\sqcup\{d\},V,U\sqcup U_d)\in \mathbf{Graph}_n$ and $d$ gives a system of cofinal sequences for $R_c$. 
\end{theorem}

\begin{proof} We apply Theorem \ref{Inv_rat_theorem} to $\mathbf{Graph}_n$. We obtain deterministic  tree $T=(\mathbf{D},V,E)\in \mathbf{Tree}_n$ and regular languages $L_e\subset (\mathbf{D}\sqcup \mathbf{D}^{-})^{\star}$, for $e\in\mathbf{C}$. We consider an automaton $\mathcal{B}$ and initial state $q_I$ and accepting state $q_c$ that recognize the language $L_c$. Note that for all $v_1,v_2\in V$ we have $(v_1,c,v_2)\in U$ iff $\mathcal{B}$ can switch from $q_I$ to $q_c$ on a path from $v_1$ to $v_2$ in $T$. 

  We consider some $R_c$-limit point $v_0$ (note that the set of all $R_c$-limit points is monadically definable). Then we consider the set $S_{v_0}$ of all points $v$ such that $T$-cone under $v$ contains some subset $R_c$-cofinal in $v_0$. $T$ is deterministic, there are only finitely many $T$-successors of any vertex and thus at least one $T$-immediate successor of a vertex from $S_{v_0}$ must lie in $S_{v_0}$.  Therefore, the set $S_{v_0}$ is infinite. We  consider the set $P_{v_0}$ of all vertices from $S_{v_0}$ that doesn't contain $v_0$ in their $T$-cone. The set of vertices $v$ from $S_{v_0}$ such that $v_0$ lies in $T$-cone under $v$ is finite because there are only finitely-many vertices above $v_0$.  Thus the set  $P_{v_0}$ is infinite.  Clearly, there is a formula $\psi(x,y)$ such that $T\models_{\mathrm{MSO}} \psi(v_0,v)$ iff $v\in P_{v_0}$.

Now let us prove that for every $v_1,v_2\in P_{v_0}$ either the vertex $v_1$ lies in the $T$-cone of $v_2$ or the the vertex $v_2$ lies in the $T$-cone of $v_1$.  For a contradiction, assume that neither the vertex $v_1$ lies in the $T$-cone of $v_2$ or the the vertex $v_2$ lies in the $T$-cone of $v_1$. We separate $T$-cone of $v_1$ on sets $A^{v_1}_t$, where $t\in \mathbf{Tp}_{\mathcal{B}}$; the set $A^{v_1}_t$ consists of all the points $v$ from $T$-cone of $v_1$ such that the type of $(v,v_1)$ is equal to $t$.  In the same fashion we  separate $T$-cone of $v_2$ on sets $A^{v_2}_t$.  Because $\mathbf{Tp}_{\mathcal{B}}$ is finite, there exist a  set  $A^{v_1}_{t_1}$ and a set  $A^{v_2}_{t_2}$ that are cofinal in $v_0$.  From Lemma \ref{type_calculation2} it follows that  the $\mathcal{B}$-type of all pairs $(u_1,u_2)$, where $u_1\in A^{v_1}_{t_1}$ and $u_2\in A^{v_2}_{t_2}$, doesn't depend on the choice of $u_1$ and $u_2$.  Thus we can $R_c$ compare the sets $ A^{v_1}_{t_1}$ and $ A^{v_2}_{t_2}$, i.e. either for all $u_1\in A^{v_1}_{t_1}$ and $u_2\in A^{v_2}_{t_2}$ we have $u_1 R_c u_2$ or for all $u_1\in A^{v_1}_{t_1}$ and $u_2\in A^{v_2}_{t_2}$ we have $u_2 R_c u_1$. Thus one of the sets is not cofinal, contradiction.

We have linear order $\sqsubset$ on $P_{v_0}$:
$$w_1\sqsubset w_2 \stackrel{\text{def}}{\iff} \mbox{ $w_1\ne w_2$ and $w_2$ is in the $T$-cone of $w_1$}.$$
Clearly, the order type of $(P_{v_0},\sqsubset)$ is $\omega$. We enumerate $P_{v_0}=\{w^{v_0}_0,w^{v_0}_1,w^{v_0}_2,\ldots\}$. Now we consider sets of vertices $B^{v_0}_0,B^{v_0}_1,B^{v_0}_2,\ldots$ that are cones under  $w^{v_0}_0,w^{v_0}_1,w^{v_0}_2,\ldots$, respectively. We consider the sequence $k^{v_0}_0<k^{v_0}_1<k^{v_0}_2<\ldots$ of all indexes $k$ such that  there exists $u\in B^{v_0}_k\setminus B^{v_0}_{k+1}$ with $uR_c v_0$.  We denote the set of all $w^{v_0}_{k^{v_0}_i}$ by $J_{v_0}$.

Now we define the desired cofinal sequence $u^{v_0}_0,u^{v_0}_1,u^{v_0}_2,\ldots$ for $v_0$. The element $u^{v_0}_i$ is $R_c$-maximal element such that $u^{v_0}_i\in B^{v_0}_{k_i}\setminus B^{v_0}_{k_i+1}$ and $u^{v_0}_iR_c v_0$.   We denote by $K_{v_0}$ the set $\{u^{v_0}_0,u^{v_0}_1,u^{v_0}_2,\ldots\}$.  

It  is easy to see that there exists a monadic formula that defines the set of all pairs $(v_0,P_{v_0})$. Thus there exists a monadic formula that define the set of all pairs $(v_0,J_{v_0})$ and therefore there exists a monadic formula that defines the set of all pairs $(v_0,K_{v_0})$. We use the late formula to define  the required  binary relation $R_d$ on $V$: $$v_1R_d v_2 \stackrel{\text{def}}{\iff} v_1\in K_{v_2}.$$

Using this definition we build monadic interpretation of $G'$ with the desired property in $T$.
\end{proof}

\begin{theorem} \label{making_Bachmann}Suppose $n$ is a number, $G=(\mathbf{C},V,U)\in \mathbf{Graph}_n$ is a well-ordering for color $c\in \mathbf{C}$, and color $d\in \mathbf{C}$ gives a system of cofinal sequences for $R_c$. Then there exists a color $e\not \in \mathbf{C}$, a set $U_e\subset V\times \{e\}\times V$ such that $G'=(\mathbf{C}\sqcup\{e\},V,U\sqcup U_e)\in \mathbf{Graph}_n$, $e$ gives a system of cofinal sequences for $R_c$, $R_e\subset R_d$, and $s_e$ have Bachmann property. 
\end{theorem}
\begin{proof}  As in the proof of the previous theorem we find deterministic tree $T=(\mathbf{D},V,E)\in \mathbf{Tree}_n$, and regular languages $L_a\subset (\mathbf{D}\sqcup \mathbf{D}^{-})^{\star}$, for $a\in\mathbf{C}$. Then we build an automaton $\mathcal{B}$, the initial states $q_c^I$, $q_d^I$, and final states states $q_c$ and $q_d$ such that $\mathcal{B}$ recognize $L_c$ on runs that starts from the state $q_c^I$ and ends at the state $q_c$ and recognize $L_d$ on runs that starts from the state $q_d^I$ and ends on the state $q_d$.  Note that for all $v_1,v_2\in V$ we have
\begin{enumerate}
 \item $(v_1,c,v_2)\in U$ iff $\mathcal{B}$ can switch from $q_I$ to $q_c$ on a path from $v_1$ to $v_2$ in $T$,
 \item $(v_1,d,v_2)\in U$ iff $\mathcal{B}$ can switch from $q_I$ to $q_d$ on a path from $v_1$ to $v_2$ in $T$.
\end{enumerate}


Let us consider some $R_c$-limit point $v_0$ . We find the point $u'$  on the path  from $v_0$ over inverse edges of $T$ to the root of $T$ such that $T$-cone of $u'$ contains infinitely many elements of $s_d$ cofinal sequence for $v_0$. We fix some ordering on $\mathbf{D}$. We consider first color $a\in D$ and $a$-successor $u_{v_0}$ of $u_{v_0}$ in $T$ such that $u_{v_0}$ exists,  $T$-cone of $u_{v_0}$ doesn't contain $v_0$, and $T$-cone of $u_{v_0}$ contains infinitely many elements of $s_d$ cofinal sequence for $v_0$. Simple check shows that such a color $a$  and such a vertex $u_{v_0}$ exists.  Thus we have found in a deterministic way a cone that contains infinitely many elements of $s_d$ cofinal sequence for $v_0$ but not $v_0$ itself.

We fix some linear ordering of $\mathbf{Tp}_{\mathcal{B}}$. We consider the least vertex pair type $t\in \mathbf{Tp}_{\mathcal{B}}$ such that there are infinitely many elements $w$ of cofinal sequence for $v_0$  in the cone of $u_{v_0}$ such that the type of $(w,u_{v_0})$ is $t$; we denote the set of all that $w$ by $P_{v_0}$.  From Lemma \ref{type_calculation1} it follows that all elements of $P_{v_0}$ were elements of $s_d$ cofinal sequence for $v_0$.

Clearly the set $B$ of all triples $(v_0, u_{v_0},P_{v_0})$ is monadically definable. Now for a given $R_c$-limit point $v_0$ we consider the set $Z_{v_0}$ of all triples  $(v_1, u_{v_1},P_{v_1})\in B$, where $v_1\ne v_0$ and $u_{v_1}$ is on the path from $v_0$ over inverse edges in $T$ to $T$-root.  Obviously there are only finitely many different $u_{v_1}$'s in elements of $Z_{v_0}$. Thus, because the set $\mathbf{Tp}_{\mathcal{B}}$ is finite, there are only finitely many different $(u_{v_1},P_{v_1})$'s in elements of $Z_{v_0}$. Therefore, because $v_1$ is always the limit of $P_{v_1}$, the set $Z_{v_0}$ is finite.  We consider the set $O_{v_0}\subset P_{v_0}$ of all  $w\in P_{v_0}$ such that for any $(v_1, u_{v_1},P_{v_1})\in Z_{v_0}$ and $w'\in P_{v_1}$ we have either $w'R_c w$, or $w'=v_0$, or $v_0 R_cw'$. Because for every $(v_1, u_{v_1},P_{v_1})\in Z_{v_0}$ there are only finitely many elements of $P_{v_1}$ that are not $R_c$-greater than $v_0$, the set $O_{v_0}$ is infinite. Clearly, the set of all pairs $(v_0,O_{v_0})$ is monadically definable in $T$. 

For all $w_1,w_2\in V$ we put $$w_1 R_e w_2\stackrel{\text{def}}{\iff}\mbox{$w_2$ is $R_c$ limit point and $w_1\in O_{w_2}$}.$$
Obviously, $R_e$ gives a system of cofinal sequences for $R_c$.

Now we show that $s_e$ have Bachmann property. We consider some $R_c$-limit points $v_0$ and $v_1$ such that $v_0 R_c v_1$ we claim that we can separate $P_{v_1}$ on two disjoint sets $F^{-}$ and $F^{+}$ such that 
\begin{enumerate}
\item we have $w R_c w'$, for all $w\in F^{-}$ and $w' \in O_{v_0}$,
\item we have $w' R_c w$, for all $w\in F^{+}$ and $w' \in O_{v_0}$.
\end{enumerate}
There are two possible cases: 1. $(v_1,u_{v_1},P_{v_1})\in Z_{v_0}$, 2. $v_0$ isn't in the $T$-cone of $u_{v_1}$. In the first case we have such a separation on $F^{-}$ and $F^{+}$ by construction of $O_{v_0}$. In the second case  from Lemma \ref{type_calculation1} it follows that  all the elements of $P_{v_1}$ are in the same $R_c$ relation to $v_0$. Because $P_{v_1}$ consists of the elements of some cofinal sequence for $v_1$, there is some $w\in P_{v_1}$ such that $v_0 R_c w$. Therefore we have $v_0 R_c w$ for all $w\in P_{v_1}$. Hence we can put $F^{-}=\emptyset$ and $F^{+}=P_{v_1}$.
It is easy to see that Bachmann property for $s_e$ follows from the claim for all $v_0$ and $v_1$.\end{proof}

Using Theorem \ref{adding_cofinals} and Theorem \ref{making_Bachmann} we obtain stronger version of  Theorem \ref{adding_cofinals}.
\begin{corollary}  Suppose $n$ is a number, $G=(\mathbf{C},V,U)\in \mathbf{Graph}_n$ is a well-ordering for color $c\in \mathbf{C}$. Then there exists a color $d\not \in \mathbf{C}$, a set $U_d\subset V\times \{d\}\times V$ such that $G'=(\mathbf{C}\sqcup\{d\},V,U\sqcup U_d)\in \mathbf{Graph}_n$ and $d$ gives a system of cofinal sequences for $R_c$ with Bachmann property. 
\end{corollary}

\section{Equivalence of Systems of Cofinal Sequences}
\label{fast_growing_section}

Suppose $(A,R)$ is a strict well-ordering  and $s\colon A^{\mathrm{lim}}\times \omega  \to A$ is a system of cofinal sequences. For an element $a\in A$ we encode finite down-paths from $a$ by finite sequences of natural numbers. We simultaneously define set $\mathbf{Path}^s_a$ of path codes and function $\rho^s_a\colon \mathbf{Path}^s_a\to A$ that maps a path code to the end of the corresponding path. Empty sequence $()$ lies in $\mathbf{Path}^s_a$ and $\rho^s_a( ())=a$. If sequence $(n_1,\ldots,n_k)$ lies in $\mathbf{Path}^s_a$ then  
\begin{enumerate}
\item if $\rho^s_a((n_1,\ldots,n_k))\in A^{\mathrm{lim}}$ then for all $m\ge 1$ the sequence $(n_1,\ldots,n_k,m)\in \mathbf{Path}^s_a$ and $\rho^s_a((n_1,\ldots,n_k,m))=s(\rho^s_a((n_1,\ldots,n_k)),m-1)$,
\item if $\rho^s_a((n_1,\ldots,n_k))$ is subsequent for $b\in A$ then the sequence $(n_1,\ldots,n_k,0)\in \mathbf{Path}^s_a$ and $\rho^s_a((n_1,\ldots,n_k,0))=b$.
\end{enumerate}
The set $\mathbf{Path}^s_a$ is the minimal set with described properties. Clearly, for every $b R a$ there  exists $p\in \mathbf{Path}^s_a$ such that $\rho^s_a(p)=b$.

For a sequence of natural numbers $(n_1,\ldots,n_k)$ we put $|(n_1,\ldots,n_k)|=n_1+\ldots+n_k+k$

Note that for a cofinal sequence systems with Bachmann property there is simple algorithm to find the path to a target point with the least $|\cdot|$. The path is defined by induction. Suppose our partial path is $(e_1,\ldots,e_k)$. We find the least $e_{k+1}$ such that $(e_1,\ldots,e_k,e_{k+1})$ encodes a path which ends either at our target point or at  a point that is above our target point; in the first case we are done in the second case we repeat the procedure. The path defined this way have the least $|\cdot|$ because every other path to the same point will go throe all intermediate points of the calculated path.

Suppose $G=(\mathbf{C},V,U)$ is a graph, and $e\not \in\mathbf{C}$. We define {\it treegraph} of $G$, it is a graph $\mathcal{T}_{e}(G)=(\mathbf{C}\sqcup \{e\},V^{+},U^{+}\sqcup S)$. The set $V^{+}$ is the set of all non-empty sequences of elements of $V$. The set $U^{+}$ consists of all edges of the form
$$((v_1,\ldots,v_k,u),c,(v_1,\ldots,v_k,w)),$$ such that there were the edge $(u,c,w)$ in $G$. The set $S$ contains all the edges of the form $$((v_1,\ldots,v_k,u),e,(v_1,\ldots,v_k,u,u)).$$

\begin{theorem}(\cite{CarWoh03}) \label{treegraph_theorem} If $G\in\mathbf{Graph}_n$ and $\mathcal{T}_{e}(G)$ is defined then $\mathcal{T}_{e}(G)\in \mathbf{Graph}_{n+1}$.
\end{theorem}

\begin{lemma} \label{path_transform_1}Suppose graph $G=(\mathbf{C},V,U)$ lies in Caucal hierarchy, $G$ is a well-ordering for color $c\in \mathbf{C}$, $d,e\in \mathbf{C}$ give systems of cofinal sequences for $R_c$ with Bachmann property, and $a\in V$. Then there exists $n$ such that for every $p_1\in\mathbf{Path}^{s_d}_a$, $p_2\in \mathbf{Path}^{s_e}_a$ with $\rho^{s_d}_a(p_1)R \rho^{s_e}_a(p_2)$  there exists $p_3\in\mathbf{Path}^{s_e}_{\rho^{s_e}_a(p_2)}$ such that $\rho^{s_e}_{\rho^{s_e}_a(p_2)}(p_3)=\rho^{s_d}_a(p_1)$ and $|p_3|\le\beth_n(|p_1|+|p_2|)$.
\end{lemma}
\begin{proof} We switch to a different graph $G'$. The set of vertices of $G'$ is $V_a$, which is the set of $a$ and all vertices $R_c$ below $a$. The set of colors of $G'$ is $\{d_0,d_1,d_2,e_0,e_1,e_2,o\}$.  We will describe how we construct $R_{d_0}$ $R_{d_1}$ and $R_{d_2}$ from $R_d$; we construct  $R_{e_0}$, $R_{e_1}$, and $R_{e_2}$ from $R_e$ in the same way and omit the description. For all $v_1,v_2\in V_a$ we put:
$$v_1R_{d_0}v_2\stackrel{\text{def}}{\iff} \mbox{$v_2$ is $R_c$-limit point and $s_d(v_2,0)=v_1$},$$
$$v_1R_{d_1}v_2\stackrel{\text{def}}{\iff} \mbox{$v_2$ is $R_c$-limit point and $s_d(v_2,1)=v_1$},$$
$$v_1R_{d_2}v_2\stackrel{\text{def}}{\iff} \mbox{$\exists n\ge 1, v_0\in V_a^{\mathrm{lim}}(s_d(v_0,n)=v_1$ and $s_d(v_0,n+1)=v_2)$},$$
$$v_1R_{o}v_2\stackrel{\text{def}}{\iff} \mbox{$v_2$ is immediate $R_c$-successor of $v_1$}.$$

Because, every $\mathbf{Graph}_n$ is closed under monadic interpretations with domain restriction \cite{CarWoh03}, the graph $G'$ lies in Caucal hierarchy. Note that graph $\mathcal{R}(G')$ is deterministic for all the colors save $d_0$ and $e_0$. It could be easily shown using the fact that $s_d$ have Bachmann property and hence for every $v\in V_a$ there is at most one $v'\in V_a^{\mathrm{lim}}$ such that $s_d(v',n)=v$, for $n\ge 1$ (the same holds for $s_e$).  Also note that the restriction of $R_c$ to $V_a$ is monadically definable in $G'$.

We consider some fresh color $h$ and the graph $\mathcal{R}(\mathcal{T}_h(G'))$. Clearly, $\mathcal{R}(\mathcal{T}_h(G'))$ is deterministic  for all the colors save $d_0$ and $e_0$. From Theorem \ref{treegraph_theorem} it follows that $\mathcal{R}(\mathcal{T}_h(G'))$ lies in Caucal hierarchy.  We consider subgraph $H$ of $\mathcal{R}(\mathcal{T}_h(G'))$ that consists of all sequences of length at most 2, i.e. root copy of $\mathcal{R}(G')$ and copies one step below it. Clearly $H$ lies in Caucal hierarchy. We consider fresh colors $l,r$, add $l$-loops to every vertex $(v_1,v_2)$ of $H$ such that $v_1 R_c v_2$, and add $r$-loops to every vertex $(v_1,v_2)$ of $H$ such that $v_2 R_c v_1$ or $v_2=v_1$; thus we obtain graph $H'$ with colors $\{d_0,d_1,d_2,e_0,e_1,e_2,o,h,d_0^{-},d_1^{-},d_2^{-},e_0^{-},e_1^{-},e_2^{-},o^{-},h^{-},l,r\}$. It is easy to see that the sets of vertices where we have added $l$-loops and $r$-loops are monadically definable in $H$ and therefore $H'$ lies in Caucal hierarchy. Also, clearly $H'$ is deterministic  for all the colors save $d_0$ and $e_0$.. We build $H''$ from $H'$ by adding color $b$, all the edges of   the form $((u),b,(u,a))$, and removing colors $d_0,d_1,d_2,e_0,e_1,e_2,o$. Clearly, $H''$ lies in Caucal hierarchy, is deterministic, and all vertices are reachable from $(a)$.

  We consider regular language 
$$L=(l\{o^{-}, e_0^{-},e_1^{-}(re_2^{-})^{\star}\})^{\star}h^{-}.$$
It is easy  to see for any two $v_1,v_2\in V_a$, $v_1 R_c v_2$ there exist exactly one path from $(v_1,v_2)$ to $(v_1)$ in $H''$ marked by some $\alpha\in L$. Moreover the corresponding $\alpha$ could be calculated from the $|\cdot|$-least path $p\in \mathbf{Path}^{s_e}_{v_2}$, $p=(n_1,\ldots,n_k)$ such that $\rho_{v_2}^{s_e}(v_2)=v_1$:
$\alpha$ is the word $\alpha_1\ldots\alpha_kh^{-}$, where, for $1\le i\le k$ the word $\alpha_i$ is equal to  the word $l o^{-}$ if $n_i=0$, is equal to the word $l e_0^{-}$ if $n_i=1$ and is equal to $le_1^{-}(re_2^{-})^{n_i-1}$ if $n_i\ge 1$. 

We consider some $p_1,p_2$ as in the lemma formulation. We consider $|\cdot|$-least $p_3\in \mathbf{Path}^{s_e}_{\rho^{s_e}_a(p_2)}$ such that $\rho^{s_e}_{\rho^{s_e}_a(p_2)}(p_3)=\rho^{s_d}_a(p_1)$. From the consideration above we see that there exists exactly one path from $(\rho^{s_d}_a(p_1),\rho^{s_e}_a(p_2))$ to $(\rho^{s_d}_a(p_1))$ marked by some $\alpha_0\in L$ and $|p_3|\le|\alpha_0|$.
It is easy to see that there is path in $H''$ from $(a)$ to $(\rho^{s_d}_a(p_1))$ of the length at most $|p_1|$, there is path in $H''$ from  $(\rho^{s_d}_a(p_1))$ to  $(\rho^{s_d}_a(p_1),\rho^{s_e}_a(p_2))$ of the length at most $|p_2|+1$. 

 Using Theorem \ref{HOPDG} we obtain some $m$-pds $\mathcal{A}$ such that $\varepsilon$-contraction of configuration graph of $\mathcal{A}$ is isomorphic to $H''$. Clearly, we can assume that initial configuration of  $\mathcal{A}$ corresponds to vertex $(a)$, because we can always modify $\mathcal{A}$ if it wasn't initially true. There is path in $\varepsilon$-contraction of configuration graph of $\mathcal{A}$ from initial configuration of $\mathcal{A}$ to the configuration corresponding to  $(\rho^{s_d}_a(p_1),\rho^{s_e}_a(p_2))$ of the length at most $|p_2|+2|p_1|+1$.   Now we apply Theorem \ref{pumping_lemma} to $\mathcal{A}$, language $L$ and the configuration corresponding to  $(\rho^{s_d}_a(p_1),\rho^{s_e}_a(p_2))$  and see that $|\alpha_0|<\beth_m(C_{\mathcal{A},L}(|p_1|+|p_2|+2))$ (note that every deterministic graph is finitely-branching and thus we can apply Theorem \ref{pumping_lemma}). Hence $|p_3|<\beth_m(C_{\mathcal{A},L}(|p_2|+2|p_1|+2))$. Hence we can choose $n$ that doesn't depend on $p_1$ and $p_2$ such that $|p_3|\le \beth_n(|p_1|+|p_2|)$. \end{proof}

The following two lemmas could be proved in the same fashion as the previous lemma with only slight modifications:
\begin{lemma} \label{path_transform_2} Suppose graph $G=(\mathbf{C},V,U)$ lies in Caucal hierarchy, $G$ is a well-ordering for color $c\in \mathbf{C}$, $d,e\in \mathbf{C}$ give systems of cofinal sequences for $R_c$ with Bachmann property, $a\in V$, and $f$ is embedding of $(V,R_c)$ onto initial segment of ordinals.  Then there exists $n$ such that for every $p_1\in\mathbf{Path}^{s_d}_a$, $p_2\in \mathbf{Path}^{s_e}_a$ with $\rho^{s_d}_a(p_1)R \rho^{s_e}_a(p_2)$  there exists $p_3\in\mathbf{Path}^{s_e}_{\rho^{s_e}_a(p_2)}$ such that $\rho^{s_e}_{\rho^{s_e}_a(p_2)}(p_3)=f^{-1}(f(\rho^{s_d}_a(p_1))+1)$ and $|p_3|\le\beth_n(|p_1|+|p_2|)$.
\end{lemma}
\begin{lemma} \label{path_transform_3} Suppose graph $G=(\mathbf{C},V,U)$ lies in Caucal hierarchy, $G$ is a well-ordering for color $c\in \mathbf{C}$, $d\in \mathbf{C}$ give systems of cofinal sequences for $R_c$ with Bachmann property, $a\in V$, and $f$ is embedding of $(V,R_c)$ onto initial segment of ordinals. Then there exists $n$ such that for every $p\in\mathbf{Path}^{s_d}_a$ there exist $v\in V^{\mathrm{lim}}\cup\{f^{-1}(0)\}$ and $k<\beth_n(|p|)$ such that $\rho^{s_d}_{a}(p)=f^{-1}(f(v)+k)$.
\end{lemma}

\begin{lemma}\label{coherent_dom} Suppose $(A,R)$ is a well-ordering and $s$ is a system of cofinal sequences with Bachmann property, $b R a\in A$,  $p\in \mathbf{Path}^s_a$, and $\rho^s_a(p)=b$. Then for all $x\ge |p|$ we have $F^s_a(x)\ge F^s_b(x)$.
\end{lemma}
\begin{proof} We prove the lemma by transfinite induction on $a$. The case of subsequent $a$ is obviously true if $a$ is successor of $b$ and trivially follows from induction assumption otherwise. If $a$ is limit then $F^s_a(x)=F^s_{s(a,x)}(x)$. Clearly the first point in path $p=(u_1,\ldots,u_k)$, i.e $s(a,u_1)$, is step-down reachable from $s(a,x)$ and hence $F^s_{s(a,x)}(y)>F^s_{s(a,u_1)}(y)$ for every $y$. Here we use the fact  that for every $c$ and step-down reachable from $c$ point $d$ we have $F^s_{c}(y)\ge F^s_{d}(y)$, for every $y$; this fact could be easily proved by transfinite induction.
\end{proof}

\begin{theorem} \label{Fast_growing_equivalency_1}Suppose graph $G=(\mathbf{C},V,U)$ lies in Caucal hierarchy, $G$ is a well-ordering for color $c\in \mathbf{C}$, $d,e\in \mathbf{C}$ give systems of cofinal sequences for $R_c$, and $s_d,s_e$ have Bachmann property. Then for every $a R_c b\in V$ there exists $N$ such that for all $x>N$ we have $F^{s_e}_b(x)> F^{s_d}_{a}(x)$.
\end{theorem}
\begin{proof} We denote by $f$ the embedding of $(V,R_c)$ onto the initial segment of ordinals. 

 Let us fix any $a\in V$. We find $n$ that is the maximal of three $n$-s provided by Lemma \ref{path_transform_1} and Lemma \ref{path_transform_2} Lemma \ref{path_transform_3}. We prove by transfinite induction on $v_2 R_c a$ that for all $v_1R_c v_2$, $p_1\in\mathbf{Path}^{s_e}_{a}$, $p_2\in\mathbf{Path}^{s_d}_{a}$ where $\rho^{s_d}_{v_0}(p_1)=v_1$ and $\rho^{s_e}_{v_0}(p_2)=v_2$ we have $F^{s_d}_{v_2}(x)>F^{s_e}_{v_1}(x)$, for all $x\ge \beth_n(|p_1|+|p_2|)$.

 There exists a path $p_3\in\mathbf{Path}^{s_e}_{v_2}$ such that $\rho^{s_e}_{v_2}(p_3)=f^{-1}(f(v_1)+1)=v_3$ and $|p_3|<\beth_n(|p_1|+|p_2|)$ by  Lemma \ref{path_transform_2}. Hence by Lemma \ref{coherent_dom} for all $y\ge \beth_n(|p_1|+|p_2|)$ we have $F^{s_e}_{v_2}(y)\ge F^{s_e}_{v_3}(y)$. If there are no limit points below $v_3$ then we are done because $F^{s_e}_{v_3}(y)=F^{s_d}_{v_3}(y)> F^{s_d}_{v_1}(x)$, for every $x$.

 Further we assume that there are limit points below $v_1$. Assume that $v_1=f^{-1}(f(v_4)+k)$, where $v_4$ is a limit point; note that from Lemma \ref{path_transform_1}  it follows that $k\le\beth_n(|p_1|)$. 

We claim that $s_d(v_4,y)R_c s_e(v_4,\beth_n(\beth_n(y)))$ for all $y>\beth_n(\beth_n(|p_1|+|p_2|))$. Indeed, let us consider some $y>\beth_n(\beth_n(|p_1|+|p_2|))$. There is  $q_1\in \mathbf{Path}^{s_d}_a$ with $\rho_a^{s_d}(q_1)=s_d(v_4,y)$  and $|q_1|<y+k+|p_1|\le y+\beth_n(|p_1|+|p_2|)$ and there is $q_2\in \mathbf{Path}^{s_e}_a$ with $\rho_a^{s_e}(q_2)=v_4$  and $|q_2|=k+|p_3|+1\le \beth_n(\beth_n(|p_1|+|p_2|))$. Therefore from Lemma \ref{path_transform_1} it follows that there exists $q_3\mathbf{Path}^{s_e}_{v_4}$ with $\rho_a^{s_e}(q_3)=s_d(v_4,y)$ and $|q_3|<\beth_n(|q_1|+|q_2|)<\beth_n(3y)<\beth_n(\beth_n(y))$. Let $m_1$ denote the first element of sequence $q_3$. Clearly, $m_1<\beth_n(\beth_n(y))$ and either $s_d(v_4,y)R_c s_e(v_4,m_1))$ or $s_d(v_4,y)= s_e(v_4,m_1))$. Hence, because  $s_e(v_4,m_1))R_cs_e(v_4,\beth_n(\beth_n(y)))$ we have $s_d(v_4,y)R_c s_e(v_4,\beth_n(\beth_n(y)))$. 

Now we see that from induction hypothesis it follows that $F^{s_d}_{v_4}(y)<F^{s_e}_{v_4}(\beth_n(\beth_n(y)))$, for all $y>\beth_n(\beth_n(|p_1|+|p_2|)$. Simple estimation shows that $F^{s_d}_{v_4}(y)<F^{s_e}_{f^{-1}(f(v_4)+1)}(y)$ for all $y>\beth_n(|p_1|+|p_2|)$. Hence $F^{s_d}_{f^{-1}(f(v_3)+i)}(y)<F^{s_e}_{f^{-1}(f(v_3)+i+1)}(y)$, for all $y>\beth_n(|p_1|+|p_2|)$ and natural $i$. Thus induction hypothesis holds for $v_2$.

Theorem clearly follows from induction hypothesis.
\end{proof}

\section{The Case of Ordinals Less than $\omega^\omega$}

 Theorem \ref{Fast_growing_equivalency_1} gives comparison of  two monadically definable systems of cofinal sequences on one structure. In general case we don't know how to use Theorem \ref{Fast_growing_equivalency_1} to compare systems of ordinal notations that arises from Caucal hierarchy to the standard system of ordinal notations for the same ordinal. But in the case of levels of fast-growing hierarchy below $\omega^{\omega}$ we can prove that ordinal notations from Caucal hierarchy give the same growth rate as the standard Cantor system of ordinal notations.

We give a definition of standard system of cofinal sequences $s_{\mathrm{st}}\colon \omega^\omega\cap \mathit{Lim} \times \omega \to \omega^\omega$ for ordinals less that $\omega^\omega$:
 $$s_{\mathrm{st}}(\alpha+\omega^{m+1},n)=\alpha+\omega^{m}n.$$
Note that from Cantor Theorem about normal forms of ordinals it follows that the definition is correct.

Note that the family of fast-growing functions $F^{s_{\mathrm{st}}}_{\alpha}(x)$, for $\alpha<\omega^{\omega}$, is an initial segment of Löb-Wainer hierarchy \cite{LobWai70}.

\begin{theorem}\label{Fast_growing_equivalency_2}  Suppose graph $G=(\mathbf{C},V,U)$ lies in Caucal hierarchy, $G$ is a well-ordering for color $c\in \mathbf{C}$, $d\in \mathbf{C}$ gives a system of cofinal sequences for $R_c$, $s_d$ have Bachmann property, and $f$ is an embedding of $(V,R_c)$ onto initial segment of ordinals. Then for all $\beta<\alpha<\omega^\omega$ there exists $N$ such that for all $x>N$ we have $F^{s_d}_{f^{-1}(\alpha)}(x)> F^{s_{\mathrm{st}}}_{\beta}(x)$ and $F^{s_{\mathrm{st}}}_{\alpha}(x)>F^{s_d}_{f^{-1}(\beta)}(x)$.
\end{theorem}
\begin{proof}  We claim that for every $\alpha_0<\omega^{\omega}$ the following relation $R_{\mathrm{st},\alpha_0}$ on $V$ is monadically definable:
$$ v_1 R_{\mathrm{st},\alpha_0} v_2\stackrel{\text{def}}{\iff} \mbox{$f(v_1),f(v_2)<\alpha_0$ and $\exists n (s_{\mathrm{st}}(f(v_2),n)=f(v_1))$}.$$
 The claim obviously follows from the fact that for every $n$  we can monadically define the set $A_n$ of all $v\in V$ such that $f(v)$ is of the form $\alpha+\omega^{n}$. We prove  monadic definability of mentioned sets by induction on $n$. The set $A_0$ is the set $$V\setminus (V^{\mathrm{lim}}\cup\{\inf_{R_c}V\}).$$ The set $A_{n+1}$ is the set of all limit points $v$ of $A_n$ such that  $v$ isn't a limit point of the set of all limit points of $A_n$.

Using Theorem \ref{Fast_growing_equivalency_1}  and the claim we straightforward conclude that  theorem holds.
\end{proof}
\bibliographystyle{plain}
\bibliography{bibliography}
\end{document}